\documentclass[12pt,reqno]{article}

\usepackage{amsmath,amsthm,amssymb,amsfonts}
\usepackage[usenames]{color}
\usepackage{graphicx}
\usepackage{verbatim}
\usepackage[inline]{enumitem}

\usepackage{tikz}
\usetikzlibrary{arrows.meta,positioning}

\usepackage{authblk}

\usepackage{hyperref}

\theoremstyle{plain}
\newtheorem{theorem}{Theorem}

\theoremstyle{definition}
\newtheorem{definition}{Definition}
\newtheorem{example}{Example}

\theoremstyle{remark}

\newtheorem*{remark*}{Remark}

\newenvironment{packed_item}{
\begin{itemize}
\setlength{\itemsep}{1pt}
\setlength{\parskip}{0pt}
\setlength{\parsep}{0pt}
}{\end{itemize}}
\ExplSyntaxOn
\NewDocumentCommand{\iqtest}{m}{
\seq_set_split:Nnn\l_tmpa_seq{,}{#1}
\begin{enumerate*}[label=\textbf{\alph*)},itemjoin=\quad\quad\quad]
\seq_map_inline:Nn \l_tmpa_seq { \item ##1 }
\end{enumerate*}}
\ExplSyntaxOff

\begin{document}

\title{Immensely Questionable Tests}

\author{Siyona Agarwal}
\author{Julian Bernhoft}
\author{Karam Gill}
\author{Yonis Gulleth}
\author{Eric Huang}
\author{Benjamin Li}
\author{Brandon Ni}
\author{Soham Samanta}
\author{Leone Seidel}
\author{Boya Yun}
\affil{PRIMES STEP}
\author{Tanya Khovanova}
\affil{MIT}
%\date

\maketitle

\begin{abstract}
We study self-referential multiple-choice tests in which the answer choices are positive integers and the question asks how many answer choices are correct. These tests quickly lead to several kinds of ambiguity and paradox. The paper presents the ideas through a dialogue between Queen Tanya and her ten sages, developing a framework for studying these immensely questionable tests.
\end{abstract}

\section{Introduction}

This paper explores the math behind self-referential multiple-choice questions, which we call \textit{Immensely Questionable tests} or \textit{IQ tests} for short. The initial idea for these tests comes from a similar puzzle described in percentages in the book \textit{Mathematical Puzzles and Curiosities} \cite{DKS}. We want to see what happens if we look at these tests using integer sequences instead of percentages. Our paper is entirely built around one important question: ``How many correct answer choices are there?'' We present our research as a story in which Queen Tanya challenges her ten sages to solve these tricky tests, with each test introducing new ideas. What starts as a simple question gets complicated very quickly. 

\section{The Story Begins}

Queen Tanya realizes on a rainy day in her castle that she has no successor for the throne. She needs to find a successor, so she comes up with a brilliant idea: IQ tests! She calls upon ten sages: Ben, Boya, Brandon, Eric, Julian, Karam, Leone, Soham, Siyona, and Yonis. Having been rudely awakened from his contemplation on an IMO problem, Sage Ben interjects, ``Why are we here in the first place?'' Queen Tanya replies, ``Silence, young Ben, you will find out soon.'' All the sages sit with worried looks on their faces, anticipating all the possibilities that could ensue. Queen Tanya stands on her throne and announces, ``I am looking for a successor.'' The Sages' blood pressure spikes. She continues, ``Don't worry, it won't be one of you.'' Sage Yonis rudely interrupts and says, ``And how will you decide?'' Queen Tanya replies, ``With an IQ test.'' Sage Brandon smirks and says, ``Why do you need us? Surely we can pass any test you design.'' Queen Tanya grins slyly and says, ``I need your help to test the tests.'' Queen Tanya hands out a pamphlet with a note: ``To try this test, you need to pay \$13.'' The test consists of the following question:

\begin{quote}
    How many correct answer choices are there?
\begin{center}
\iqtest{2,3,4,4}
\end{center}
\end{quote}

By the way, this question is inspired by a similar puzzle described as a percentage puzzle in the book \cite{DKS}, where it is called an IQ test. So Tanya kept the name.

\begin{quote}
If you choose the answer to this question at random, what is the probability you will be correct?
\begin{center}
\iqtest{25\%,50\%,60\%,25\%}
\end{center}
\end{quote}

\section{Test 1: (2,3,4,4)}

The sages haven't seen the book and immediately start discussing the test. 

Sage Soham takes the lead: ``There are no correct answers here. But I'm sure there will be more tests, and they will be more difficult. We should prepare. Let us denote the number of options by $K$. We can also assume that the options are in non-decreasing order. So we denote the test as $(c_1,c_2,\dots, c_K)$, where $c_1 \le c_2 \le \cdots \le c_K$.''

Sage Leone continues: ``I think the cost is not random; it is the sum of all the options. Let us denote it by $C$. By our definition, we have
\[c_1 + c_2 + \dots + c_K = C.\]''

Sage Soham says, ``The cost of such a test is the sum of all the values. Thus, such a test corresponds to a \textit{partition} of its cost. By the way, the number of partitions of $n$ is denoted by $p(n)$.''

Sage Yonis says while yawning, ``Let's see what awaits us tomorrow.''

\section{Test 2: (1,1,2,2)}

The next day, Queen Tanya gives a new test: $(1,1,2,2)$ of cost $6$.

Sage Siyona is quick to answer, ``The answer should be two because two appears twice.'' 
Sages Leone and Boya nod their heads, agreeing. Sage Karam yawns and says, ``Is this all you want from us today?'' Tanya chuckles and says, ``I will be back tomorrow.''

Sage Brandon exclaims, ``Who knows how hard these questions will get? We must formulate a general strategy to solve these tests.'' Sage Eric suggests, ``If the correct answer is $k$, then it must appear $k$ times, right? What more is here?''

Sage Leone mentions, ``Notice that we  have confirmed that the cost is the sum of the options.''

\section{Test 3: (1,2,2)}

The next day, Queen Tanya brings a new test: $(1,2,2)$. Sage Karam says, ``One appears once, so one must be the correct answer.'' To that, Sage Julian replies, ``Two would actually be the correct answer since two appears twice.'' 
Sage Karam comments, ``Oh no, there are two correct answers: $1$ and $2$.''

Sage Soham replies, ``Well, that wouldn't work since we can't have two correct answers.'' To that, Sage Ben replies, ``Who said we couldn't have two correct answers?''

Sage Soham says, ``So since both $1$ and $2$ are correct, then we have $3$ correct options, so $3$ is the actual answer.'' Sage Karam then responds, ``But there is no answer $3$ on the test.''

Sage Eric states, ``In general, I don't think we can have two correct answers since if we assume answers $a$ and $b$ are both correct, then we have $a+b$ correct answers so neither $a$ nor $b$ is correct.'' Sage Boya then questions, ``All of this is a mess, so are both $a$ and $b$ correct, or are none of them correct?''

Sage Julian says, ``Let's adjust our definitions. The word correct is too strong.''

\begin{definition}
    For every answer $a$, we denote the number of times it appears in the test by $m_a$ and call it the \textit{multiplicity} of $a$. We call an answer $a$ \textit{valid} if it appears exactly $a$ times: $m_a = a$. We call the sum of the multiplicities of the valid answers a \textit{validity score} and denote it as $V$.
\end{definition}

Sage Julian clarifies, ``We don't count the total number of distinct answers; we count all options. In today's test, the validity score is $3$.'' Sage Karam goes further, ``The validity score is the sum of all the multiplicities of distinct valid answers.''

Sage Ben adds, ``Let us define the multiplicity of a set of answers. One day we might need it.''

\begin{definition}
    Given a set of answers $S$, we denote the total number of times the answers in $S$ appear as options as $m_S$ and call it the \textit{multiplicity} of $S$: $m_S=\sum_{a\in S} m_a$.
\end{definition}

Sage Yonis continues, ``So you are saying that the test is correct, in other words, has one valid answer, if the block of rows corresponding to the answer $a$ forms an $a \times a$ square, right? In case you didn't know, a Young diagram is a shape that looks like a rectangular grid with a bottom-right part cut off. It corresponds to a partition, where the $i$th row from the top has the same number of cells as the $i$th largest part.''

Sage Karam gets busy and makes diagrams for Tests $2$ and $3$ in Figures~\ref{fig:test2} and \ref{fig:test3}, respectively. The valid answers are gray. Test 2 has one valid answer marked as a gray 2-by-2 square. In Test 3, all answers are valid, so the whole diagram is gray.

\begin{figure}[ht!]
		\centering
		\begin{minipage}{0.45\textwidth}
			\centering
			\begin{tikzpicture}[scale=0.6]
				\draw[fill=gray!30, thick] (0,2) rectangle (2,4);
				\draw (0,3) -- (2,3);
				\draw (1,2) -- (1,4);
				\draw[fill=gray!0, thick] (0,0) rectangle (1,2);
				\draw (0,1) -- (1,1);
			\end{tikzpicture}
			\caption{Test 2: $(1,1,2,2)$}
            \label{fig:test2}
		\end{minipage}\hfill
		%young diagram for test 3
		\begin{minipage}{0.45\textwidth}
			\centering
			\begin{tikzpicture}[scale=0.6]
				\draw[fill=gray!30, thick] (0,1) rectangle (2,3);
				\draw (0,2) -- (2,2);
				\draw (1,1) -- (1,3);
				\draw[fill=gray!30, thick] (0,0) rectangle (1,1);
			\end{tikzpicture}
			\caption{Test 3: $(1,2,2)$}
            \label{fig:test3}
		\end{minipage}
	\end{figure}

Sage Soham says, ``Let us look at the big picture. The first test didn't have any valid answers. The second test had one valid answer, and there was no controversy there. This test has two valid answers. Maybe tomorrow we will have three valid answers.''

Sage Leone suggests: ``Let's have some definitions.''

\begin{definition}
We call a test \textit{unsolvable} if there are no valid answers. We call a test \textit{solvable} if there is at least one valid answer. We call a test \textit{monosolvable} if there is exactly one valid answer. We call a test \textit{polysolvable} if it has more than one valid answer. In addition, we call a test $k$-\textit{solvable} if it has exactly $k$ valid answers.
\end{definition}

Sage Julian notices, ``In particular, monosolvable tests are also $1$-solvable. Now, all the tests we saw are defined.''

\begin{example}
The first test $(2,3,4,4)$ is unsolvable. The second test $(1,1,2,2)$ is monosolvable with the only valid answer $2$. Today's test $(1,2,2)$ is polysolvable with valid answers $1$ and $2$, and, more precisely, it is $2$-solvable.
\end{example}

Sage Eric adds, ``What should we call a correct test or a correct answer, for that matter? An unsolvable test cannot be correct. If the test is polysolvable, then the validity score $V$ is greater than any of the valid answers, implying that the valid values are incorrect. Thus, we can argue that we should call an answer \textit{correct} if it is the only valid answer of a monosolvable test.''

Sage Ben mentions, ``By the way, we can describe a valid answer in terms of partitions and Young diagrams. Suppose we have a partition of $C$. We can assume that the rows of the Young diagram correspond to answer options. We can split the Young diagram into rectangles of different lengths. And the answer is valid when the corresponding rectangle is a square.''

Sage Yonis says, ``Let's discuss more details tomorrow after a good night's sleep.''

\section{Test 4: (1,2,2,3,3,3,3)}

The sages don't really get a good night's sleep, worrying about the next test. Tired, the sages wait to see what comes next. However, the next day, the sages are presented with a much harder question, and quite a deceptive one it is. The sages walk into the main hall, and Queen Tanya, with a cunning smile spreading across her face, asks the question, ``What is the number of correct answers to this question?'' The sages wait eagerly to hear what the answer choices are. Queen Tanya reads the test out loud: $(1,2,2,3,3,3,3)$.

Sage Soham declares, ``This seems like the one yesterday: Both $1$ and $2$ are valid answers.'' Sage Eric continues, ``Not quite; the total number of valid options is also on the test.'' Sage Yonis exclaims, ``Our biggest fear has come true. Tanya has given us a paradox!''

Sage Siyona, dreading the punishment for giving wrong answers, commands, ``Get your act together, everyone. Let's find some consensus.'' Chaos breaks out as all the sages are worried. 

Sage Boya shouts, ``There is nothing to fear. Let's give a name to such a test. Let's call such tests paradoxical.''

\begin{definition}
Let $a_1<a_2<\cdots<a_k$ be all the distinct valid answers of a test, where $k\geq 2$. Its validity score is
\[
V=\sum_{i=1}^k m_{a_i}=\sum_{i=1}^k a_i.
\]
The test is called \emph{paradoxical}, or a \emph{para-test}, if $V$ occurs among its answer choices.
\end{definition}

Sage Brandon says, ``I like the idea of naming everything. Let's give a name to all Tanya's tests: \textit{Immensely Questionable tests} --- IQ tests for short.''

Sage Eric jumps in, ``An answer $a$ is para-valid if and only if, when the rows corresponding to the valid answer choices are all put together and set to length $a$, a square is formed.''

\begin{example}
    Consider test $(1, 2, 2, 3)$. The answer $3$ is para-valid because when the rows of the valid answers $\{1, 2\}$ are taken, and all of their lengths are set to $3$, a square is formed, as seen in Figure~\ref{fig:paravalid}.

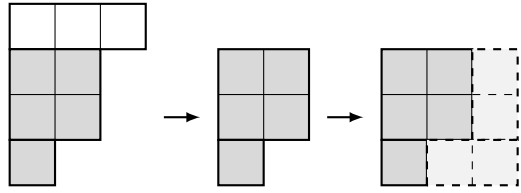
\begin{figure}[ht!]
    \centering
\begin{tikzpicture}[scale=0.6]
  \draw[fill=white, thick] (0,3) rectangle (3,4);
  \foreach \x in {1,2} \draw (\x,3) -- (\x,4);
  \draw[fill=gray!30, thick] (0,1) rectangle (2,3);
  \draw (0,2) -- (2,2);
  \draw (1,1) -- (1,3);
  \draw[fill=gray!30, thick] (0,0) rectangle (1,1);
  \draw[-latex, thick] (3.4,1.5) -- (4.2,1.5);
  \begin{scope}[shift={(4.6,0)}]
    \draw[fill=gray!30, thick] (0,1) rectangle (2,3);
    \draw (0,2) -- (2,2);
    \draw (1,1) -- (1,3);
    \draw[fill=gray!30, thick] (0,0) rectangle (1,1);
  \end{scope}
  \draw[-latex, thick] (7.0,1.5) -- (7.8,1.5);
  \begin{scope}[shift={(8.2,0)}]
    \draw[fill=gray!30, thick] (0,1) rectangle (2,3);
    \draw (0,2) -- (2,2);
    \draw (1,1) -- (1,3);
    \draw[fill=gray!30, thick] (0,0) rectangle (1,1);
    \draw[fill=gray!10, dashed, thick] (2,1) rectangle (3,3);
    \draw[dashed] (2,2) -- (3,2);
    \draw[fill=gray!10, dashed, thick] (1,0) rectangle (3,1);
    \draw[dashed] (2,0) -- (2,1);
  \end{scope}
\end{tikzpicture}
    \caption{A para-valid answer from the point of view of a Young diagram}
    \label{fig:paravalid}
\end{figure}
\end{example}

\section{Test 5: (1,2,2,3,3,3,3,4,4)}

The next day, Queen Tanya comes back and asks, ``Are you fit to be my sages?'' She gives the next test: $(1,2,2,3,3,3,3,4,4)$.

Sage Soham asks, ``What is going on here? My goodness, is it possible that the validity score is on the test, and its multiplicity is also on the test?''

Sage Brandon replies, ``We have to introduce the notion of time: different time steps. Let us denote the total number of valid options on the test, not by $V$, but by $V_1$.''

Sage Soham clarifies, ``To be concise, we should call it the \textit{first validity score}.''

Sage Brandon continues, ``In today's test, the first validity score $V_1$ is on the test.''

Sage Boya gives a formal definition of a para-valid answer.

\begin{definition}
For a polysolvable test, if $V_1$ occurs among the answer choices, then the answer value $V_1$ is called \emph{para-valid}.
\end{definition}

Sage Boya continues, ``In today's test, the answer $3$ is para-valid. Let us denote the multiplicity of the para-valid answer by $V_2$. In yesterday's and today's tests, the multiplicity $V_2 = 4$.''

Sage Yonis exclaims, ``But there is a huge difference. Yesterday, $V_2$ wasn't one of the answers, and today it is. Suppose the answer $V_2$ is on the test. Then we call it \textit{para-para-valid}, and denote its multiplicity as $V_3$. We should expand our definition.''

The sages stare at the word para-para-valid. Somewhere in the castle, a dictionary quietly resigns.

\begin{definition}
     For polysolvable tests, set $V_1$ equal to the validity score. If $V_i$ occurs as an answer, call this answer \textit{para$^i$-valid} and define $V_{i+1}=m_{V_i}$. If $V_i$ doesn't occur as an answer, the process terminates.
\end{definition}

Sage Leone dreams aloud, ``Can we have a para-para-para-valid answer? My goodness, can this continue indefinitely?''

Sage Karam has an AHA moment and exclaims, ``I can figure out an example.''

\begin{example}
    Consider a test with an infinite number of options:
    \[(1,\ 2,\ 2,\ 3,\ 3,\ 3,\ 3,\ 4,\ 4,\ 4,\ 4,\ 4,\ \ldots).\]
    This test has one $1$, two $2$s, and $a+1$ copies of the answer $a$ for any positive integer $a > 2$. The valid answers are $1$ and $2$. We can find that $V_i=i+2$ for positive $i$.
\end{example}

Queen Tanya comments, ``Be careful with infinite tests. Do you require that every answer value have finite multiplicity and that the initial validity score is finite?'' The sages continue discussing the infinite possibilities offered by infinite tests.

That evening, Queen Tanya calls the sages into the hall and tells them, ``I noticed that you guessed a few more tests that I prepared for you. Maybe I shouldn't chop off your heads after all this is over.'' The sages nod their heads, which seem to be safe so far, and retire for the evening.

\section{Test 6: (1,2,2,3)}

The next day, Queen Tanya brings a new test: $(1,2,2,3)$. It has a low cost, a meager \$8.

The sages are excited because the test looks simpler than the previous ones.

Sage Leone exclaims, ``Is this the cheapest paradoxical test?''

Sage Boya immediately starts coding, then pronounces, ``The total number of para-tests, namely tests that have at least one para-valid answer, is the following sequence as a function of cost, starting from index $0$:''
\[0,\ 0,\ 0,\ 0,\ 0,\ 0,\ 0,\ 0,\ 1,\ 0,\ 0,\ 1,\ 1,\ 1,\ 3,\ 2,\ \ldots.\]
%0, 0, 0, 0, 0, 0, 0, 0, 1, 0, 0, 1, 1, 1, 3, 2, 4, 4, 7, 7, 13, 13, 20, 22, 32, 36, 52, 59, 81, 93, 122, 143, 190, 219, 285, 335, 425, 497, 628, 734, 914, 1077, 1325, 1556, 1909, 2234, 2719, 3187, 3847, 4507, 5418, 6329, 7566, 8839, 10508, 12255, 14523, 16898, 19947, 23187, 27254, 31629, 37080, 42942, 50186, 58044, 67619, 78088, 90750, 104605, 121264, 139582, 161408, 185499, 214064, 245602, 282825, 324053, 372384, 426042, 488693, 558278, 639223, 729258, 833518
Sage Karam looks at the results closely, ``Yes, it is confirmed that the cheapest para-test has cost $8$, while the next cheapest one has cost $11$. I wonder what the test of cost $11$ might be.'' Sage Ben immediately calculates, ``It should be $(1,2,2,3,3)$.''

Meanwhile, Leone programs and says, ``I get a different sequence: it only has $2$ tests for cost $14$. I found the tests $(1,2,2,3,6)$ and $(1,3,3,3,4)$.'' Karam adds, ``Oh, I see what was happening! The previous sequence included $(1,2,2,3,3,3)$, which is $3$-solvable. But it is easy to make a mistake and look at this test as $2$-solvable with $1$ para-valid answer.'' Sage Siyona stares at the test, ``Test $(1,2,2,3,3,3)$ is different from the other ones. If $1$ and $2$ were the only initially valid answers, $3$ would be para-valid, but here $3$ is one of the initial valid answers!''

Everyone now decides that the correct sequence should be 
\[0,\ 0,\ 0,\ 0,\ 0,\ 0,\ 0,\ 0,\ 1,\ 0,\ 0,\ 1,\ 1,\ 1,\ 2,\ 2,\ \ldots.\]
%0, 0, 0, 0, 0, 0, 0, 0, 1, 0, 0, 1, 1, 1, 2, 2, 4, 4, 6, 6, 13, 12, 18, 20, 29, 34, 46, 53, 75, 86, 113, 132, 174, 199, 262, 308, 387, 455, 574, 671, 842, 986, 1211, 1425, 1748, 2044, 2494, 2917, 3523, 4124

Sage Julian interrupts, ``You guys are getting excited about the money, but what about Test 6 itself? What can we say about it?''

Sage Eric remarks, ``What is interesting about getting to a valid answer as the $i$-th validity score is that the validity score stops changing after that. The validity score stabilizes.''

Sage Brandon continues, ``We have started looking for correct tests. I think this is another example where we can call a test correct, as it feels less paradoxical than other tests.''

Sage Siyona suggests, ``Let us call a polysolvable test \textit{eventually correct} if the para-process eventually reaches a valid answer.''

At this point, the sages get very tired and retire to their rooms, except for Sage Boya, who continues coding, ignoring everything else happening in the room. ``Hey guys, I have made Table~\ref{tab:para-tests-by-cost}, which gives the numbers of $k$-solvable para-tests of cost $C$, where I omit trailing zeros in each column.''

\begin{table}[ht!]
\centering
\begin{tabular}{|c|c|c|c|c|c|c|c|c|c|c|c|c|c|c|c|}
\hline
$k\backslash C$ & 8 & 9 & 10 & 11 & 12 & 13 & 14 & 15 & 16 & 17 & 18 & 19 & 20 & 21 & 22 \\ \hline
2 & 1 & & & 1 & 1 & 1 & 2 & 2 & 4 & 4 & 6 & 6 & 12 & 12 & 18 \\ \hline
3 & & & & & & & & & & & & & 1 & & \\ \hline
4 & & & & & & & & & & & & & & & \\ \hline
\end{tabular}
\caption{Numbers of $k$-solvable para-tests of cost $C$. Blank entries represent zero.}
\label{tab:para-tests-by-cost}
\end{table}

But the room is empty.

\section{Test 7: (1, 2, 2, 3, 3, 3, 3, 4, 4, 4)}

Queen Tanya's next test is $(1,2,2,3,3,3,3,4,4,4)$.

``What do we have here?'' says Leone, ``The valid answers are again $1$ and $2$. This means $V_1=3$, making $3$ para-valid. Since $3$ appears four times,
we get $V_2=4$, making $4$ para-para-valid. Calculating further, we get $V_3=3$, $V_4=4$, and so on, alternating between $3$ and $4$. Thus, unlike the previous examples, the para-process neither stops nor stabilizes: it enters a loop of length $2$.''

``Wonderful,'' says Sage Julian. ``The test is now thinking in circles. That makes two of us. I can make a diagram, see Figure~\ref{fig:test7-loop}.''

\begin{figure}[ht!] \centering \begin{tikzpicture}[ process/.style={circle, draw, thick, minimum size=1cm, align=center}, setnode/.style={rectangle, draw, thick, rounded corners, minimum height=0.8cm, align=center}, arrow/.style={-Latex, thick} ] \node[setnode] (valid) {$\{1,2\}$\\ valid answers}; \node[process, right=2.4cm of valid] (three) {$3$}; \node[process, right=2.4cm of three] (four) {$4$}; \draw[arrow] (valid) -- node[above] {$V_1$} (three); \draw[arrow] (three) -- node[above] {$V_2=m_3$} (four); \draw[arrow, bend left=35] (four) to node[below] {$V_3=m_4$} (three); \draw[arrow, bend left=35] (three) to node[above] {$V_4=m_3$} (four); \end{tikzpicture}
\caption{The para-process for Test 7 enters the cycle $3 \leftrightarrow 4$.} 
\label{fig:test7-loop} 
\end{figure}
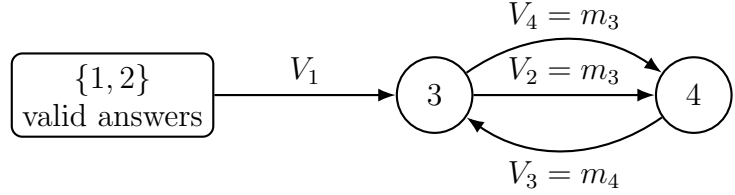

Sage Leone argues, ``We have now seen a test that stops, a test that stabilizes, and a test that runs in circles. This is also a fair description of our meetings.''

\begin{theorem}
\label{thm:paracases}
    Given a para-test where each answer value have finite multiplicity, when we continue the para-process, we end up in one of the following cases:
    \begin{enumerate}
    \item We reach a validity score that is not one of the answers;
    \item The validity scores never repeat, which can only happen for infinite tests;
    \item If the validity score repeats, it starts cycling. Namely, if $i$ is the smallest index for which there exists $j < i$ such that $V_i = V_j$, then for any positive $k$, we have $V_{i+k} = V_{j+k}$.      
    \end{enumerate}
\end{theorem}

\begin{proof}
    If we get to the validity score that is not one of the answers, we stop. Suppose this is not the case.
    
    If the validity scores never repeat, then the number of validity scores is infinite, implying that the test itself is infinite. For a finite test, every nonterminal $V_i$ is one of the finitely many answer values, so some validity score must eventually repeat.
    
    Suppose the validity score repeats, and $i$ is the smallest index such that $V_i$ repeats a previous validity score, that is, $V_i = V_j$, for $j < i$. Because the validity score $V_{n+1}$ is determined solely by $V_n$ and the fixed multiplicities in the test, the equation $V_i = V_j$ implies $V_{i+1} = V_{j+1}$, and by induction, $V_{i+k} = V_{j+k}$ for all $k \ge 0$. This confirms that the sequence enters a cycle of length $i - j$.
\end{proof}

The proof is short enough that everyone trusts it, which is unusual. Queen Tanya looks pleased, which everyone interprets as a positive sign about the survival of their necks.

Sage Brandon concludes, ``So, we have three options: stop, continue forever, or get into a loop, right?'' Sage Julian counters, ``Yes, but the theorem doesn't emphasize the most interesting case, when the loop has length $1$. In other words, when we stumble on a valid answer.''

\section*{Break}

The next day, Queen Tanya announces, ``My dear sages, you have done great work. I will allow you all to continue serving as my sages. Your brilliant and beautiful heads will stay in place.''

Everyone relaxes and starts playing Dungeons and Dragons. Sage Julian comments, ``Something bugs me. The first validity score is different from the other validity scores. If the validity score $V_i$ for $i > 1$ matches an answer in the test, it is always a single answer, implying that such $V_i$ is also a multiplicity. However, in our paradoxical tests, we start with several valid answers, and the first validity score is the sum of these multiplicities. So, theoretically, $V_1$ is different from the rest.''

Other sages whisper, ``Shh, let us rest.''

\begin{center}
    \includegraphics[scale=0.15]{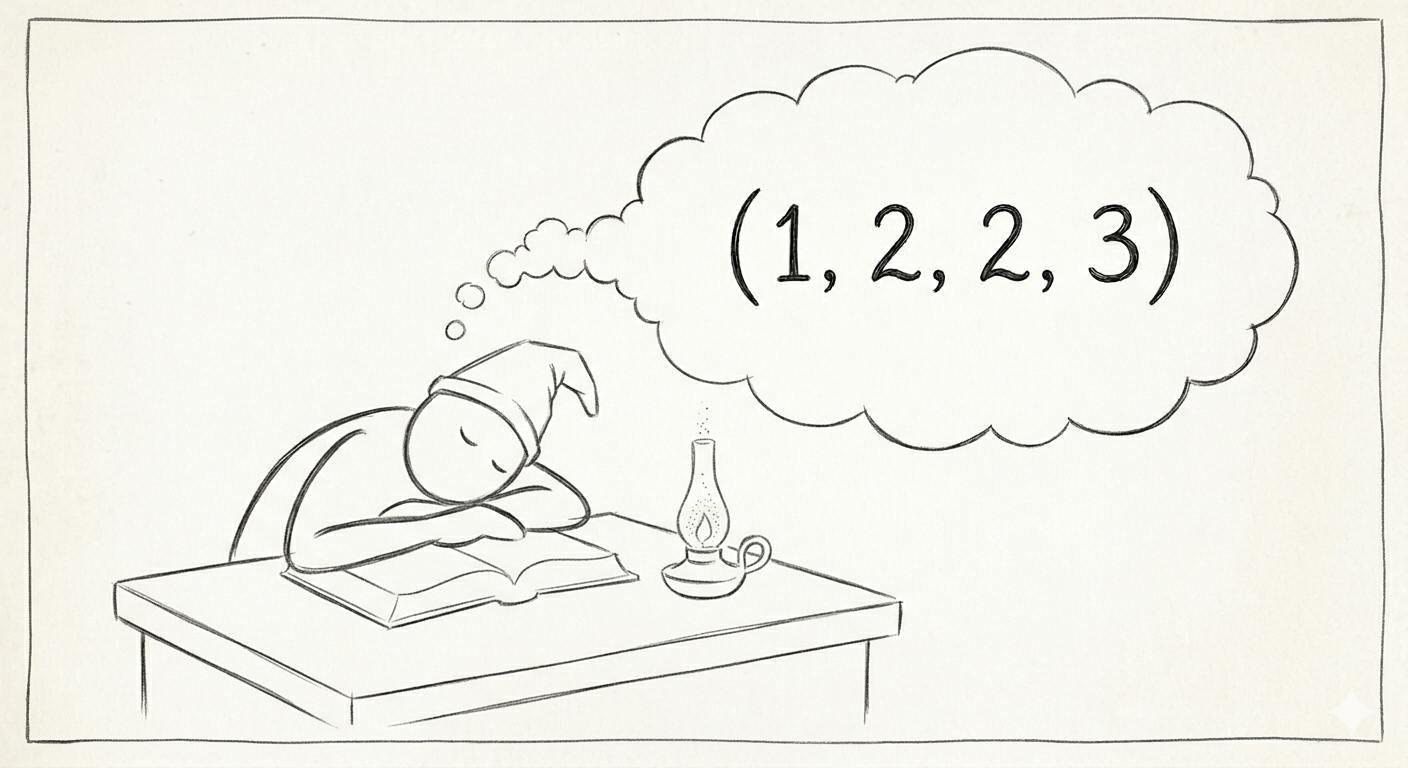}
\end{center}

Sage Brandon still can't stop thinking about all this. He murmurs to himself, ``Consider a paradoxical test $(1, 2, 2, 3)$, which currently has $1$ and $2$ as valid answers, and $3$ as a para-valid answer. 

Here, the way that an answer is judged to be valid is inconsistent, as the valid answers $1$ and $2$ count themselves when deciding if they are valid or not. However, the para-valid answer $3$ does not count itself when deciding whether it is valid or not. It uses the validity scores of the other answers. Hmm, this is confusing.''

Tanya goes to her wing of the palace, thinking that her job is done. She is mistaken.

\section{Meta-tests}

The next day, Queen Tanya is relaxing at her personal spa. But the sages come to her with a test of their own: $(1, 2)$.

Tanya says, ``So, this is easy. One valid answer, and this is it. This is a 1-solvable test.''

Sage Brandon replies, ``In all previous tests, former valid answers were no longer carried forward. What if we allow them to stay? I actually wondered about this last night.''

Sage Karam adds, ``If the answer $1$ stays valid, then together with $2$, we get two valid answers here. So the answer $2$ should be valid!''

Sage Siyona declares, ``This would be too confusing. Let's give it a different name. Let's call $2$ meta-valid. Let me give you another example.'' Sage Julian makes a diagram.

\begin{example}
\label{ex:new-meta}
Consider a test $(1, 2, 2, 4)$. We see that $1$ and $2$ are valid. If we add $4$ to the group of valid answers, the total multiplicity becomes $4$, making $4$ meta-valid.

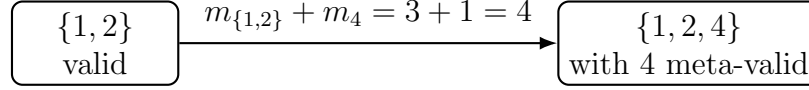
\begin{figure}[ht!] \centering 
\begin{tikzpicture}
[ statenode/.style={rectangle, draw, thick, rounded corners, minimum height=0.9cm, minimum width=2.2cm, align=center}, arrow/.style={-Latex, thick} ] \node[statenode] (valid) {$\{1,2\}$\\ valid}; \node[statenode, right=5cm of valid] (meta) {$\{1,2,4\}$\\ with $4$ meta-valid}; \draw[arrow] (valid) -- node[above] {$m_{\{1,2\}}+m_4=3+1=4$} (meta); \end{tikzpicture}
\caption{In the test $(1,2,2,4)$, the valid answers $1$ and $2$ lead to the meta-valid answer $4$.}
\label{fig:meta-process} \end{figure}
\end{example}

``Aha,'' says Brandon, ``we have a new definition.''

\begin{definition}
An answer $a$ that is not already valid is called \textit{meta-valid} if $m_a$ plus the validity score equals $a$: $m_a + V_1 = a$.
\end{definition}

Sage Eric jumps in, ``An answer $a$ is meta-valid if and only if, when the rows corresponding to the valid answer choices and the rows of length $a$ are all put together and set to length $a$, a square is formed.''

\begin{example}
    Consider test $(1, 2, 2, 4)$. The answer $4$ is meta-valid because when the rows of the valid answers $\{1, 2\}$ are taken, along with the row of length $4$, and all of their lengths are set to $4$, a square is formed, as seen in Figure~\ref{fig:metavalid}.

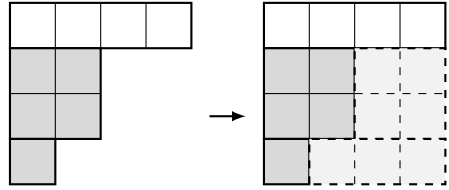
\begin{figure}[ht!]
    \centering
    \begin{tikzpicture}[scale=0.6]
  \draw[fill=white, thick] (0,3) rectangle (4,4);
  \foreach \x in {1,2,3} \draw (\x,3) -- (\x,4);
  \draw[fill=gray!30, thick] (0,1) rectangle (2,3);
  \draw (0,2) -- (2,2);
  \draw (1,1) -- (1,3);
  \draw[fill=gray!30, thick] (0,0) rectangle (1,1);
  \draw[-latex, thick] (4.4,1.5) -- (5.2,1.5);
  \begin{scope}[shift={(5.6,0)}]
    \draw[fill=white, thick] (0,3) rectangle (4,4);
    \foreach \x in {1,2,3} \draw (\x,3) -- (\x,4);
    \draw[fill=gray!30, thick] (0,1) rectangle (2,3);
    \draw (0,2) -- (2,2);
    \draw (1,1) -- (1,3);
    \draw[fill=gray!30, thick] (0,0) rectangle (1,1);
    \draw[fill=gray!10, dashed, thick] (2,1) rectangle (4,3);
    \draw[dashed] (2,2) -- (4,2);
    \draw[dashed] (3,1) -- (3,3);
    \draw[fill=gray!10, dashed, thick] (1,0) rectangle (4,1);
    \foreach \x in {2,3} \draw[dashed] (\x,0) -- (\x,1);
  \end{scope}
\end{tikzpicture}
    \caption{A meta-valid answer from the point of view of a Young diagram}
    \label{fig:metavalid}
\end{figure}
\end{example}

Sage Soham comments, ``By the way, a meta-valid answer can never be valid.''

Sage Brandon proclaims, ``This looks like a new type of test, and we call it a \textit{meta-test}, because the new answer itself participates in the count. Did you notice that today we have the cheapest test with a meta-answer?''

\begin{definition}
   A \textit{meta-test} is a test with at least one valid answer and at least one meta-valid answer.
\end{definition}

Sage Leone writes a program to calculate the number of meta-tests given the cost and gets the following sequence starting from index 0:
\[0,\ 0,\ 0,\ 1,\ 0,\ 0,\ 1,\ 3,\ 1,\ 4,\ 3,\ 6,\ 8,\ 12,\ \dots.\]
%0, 0, 0, 1, 0, 0, 1, 3, 1, 4, 3, 6, 8, 12, 13, 25, 22, 34, 42, 60, 65, 105, 112, 157, 183, 241, 285, 385, 440, 579, 686, 874, 1020, 1310, 1527, 1921, 2269, 2798, 3295, 4087, 4765, 5852, 6883, 8329, 9772, 11827, 13807, 16595, 19424, 23179

Leone says, ``For cost 7, we have a jump in the number of tests. These are tests $(1,2,4)$, $(1,3,3)$, and $(2,2,3)$.''

Sage Leone mentions, ``This proves that we got the cheapest meta-test today. The second cheapest costs $6$. I wonder what the test is.'' Sage Soham calculates, ``It is $(1,2,3)$''.

Sage Ben says, ``The cheapest para-test is more expensive than the cheapest meta-test. Hmm.''

Sage Boya, who has apparently been coding under the table since breakfast, interrupts with another program, which counts the number of meta-tests of cost $C$ with exactly $k$ valid answers. The results are shown in Table~\ref{tab:meta-tests-by-cost}.

\begin{table}[ht!]
\centering
\begin{tabular}{|c|c|c|c|c|c|c|c|c|c|c|c|c|c|c|c|c|}
\hline
$k\backslash C$ & 3 & 4 & 5 & 6 & 7 & 8 & 9 & 10 & 11 & 12 & 13 & 14 & 15 & 16 & 17 & 18 \\ \hline
1 & 1 & & & 1 & 3 & 1 & 3 & 3 & 6 & 7 & 12 & 12 & 21 & 21 & 31 & 39 \\ \hline
2 & & & & & & & 1 & & & 1 & & 1 & 4 & 1 & 3 & 3 \\ \hline
3 & & & & & & & & & & & & & & & & \\ \hline
\end{tabular}
\caption{Numbers of meta-tests of cost $C$ with exactly $k$ valid answers. Blank entries represent zero.}
\label{tab:meta-tests-by-cost}
\end{table}

The room is silent. This is either because the table is impressive or because nobody wants to be asked to explain it.

Sage Brandon wonders, ``But if answer $2$ also becomes valid, wouldn't answer $1$ stop being valid because there would now be two valid answers?''

\section{Meta-meta- and so on}

The next day, Queen Tanya receives a new test $(1, 2, 3, 4)$.

Tanya looks at it with suspicion and says, ``Interesting, here $1$ is valid, but then $2$ could be considered valid because together with a valid $1$, it is meta-valid. Thus, we have $2$ valid answers, and together with $3$, there are $3$ valid answers, and so forth. In this example, the answer $3$ is, in a sense, meta-meta-valid, and the answer $4$ is meta-meta-meta-valid.''

Sage Eric adds, ``I suggest making a more compact notation. We can use \textit{meta$^j$-valid} for an answer that has $j$ meta prefixes.''

\begin{definition}
Let $A_{0}$ be the set of valid answers: $A_{0} = \{a : m_{a} = a\}$.
Let $V_1$ be its validity score:
\[V_{1} = m_{A_0} = \sum_{a \in A_{0}} m_{a} = \sum_{a \in A_{0}} a.\]
For $j \ge 1$, suppose that the sets $A_i$ of meta$^i$-valid answers for $0 \le i < j$ have already been defined. Then define:
\[V_{j} = m_{A_{0} \cup A_{1} \cup \dots \cup A_{j-1}} = \sum_{a \in A_{0} \cup A_{1} \cup \dots \cup A_{j-1}} m_{a}\]
\[A_{j} = \{a \notin A_{0} \cup A_{1} \cup \dots \cup A_{j-1} : m_{a} + V_{j} = a\}.\]
An answer $a \in A_{j}$ is called \textit{meta$^j$-valid}. Thus, the answers in $A_{0}$ are valid, the answers in $A_{1}$ are meta-valid, the answers in $A_{2}$ are meta$^2$-valid, and so on.
\end{definition}

Sage Eric whispers, ``At this rate, the prefixes will be longer than the tests.''

\begin{example}
Consider today's test $(1, 2, 3, 4)$. Here, $m_{1} = m_{2} = m_{3} = m_{4} = 1$. Thus, $A_{0} = \{1\}$, so $V_{1} = m_{1} = 1$. Since $m_{2} + V_{1} = 1 + 1 = 2$, we have $2 \in A_{1}$, so $2$ is meta-valid. Next, since $V_{2} = m_{1} + m_{2} = 2$ and $m_{3} + V_{2} = 1 + 2 = 3$, we have $3 \in A_{2}$, so $3$ is meta$^2$-valid. Similarly, since $V_{3} = m_{1} + m_{2} + m_{3} = 3$ and $m_{4} + V_{3} = 1 + 3 = 4$, we have $4 \in A_{3}$. Therefore, $4$ is meta$^3$-valid. 
\end{example}

Sage Soham continues, ``And the cheapest test that contains a meta$^j$-valid answer is the test $(1,2,\dots,j+1)$ with cost $\frac{(j+1)(j+2)}{2}$. Tada!''

Sage Julian says, ``Let's have an infinite test: $(1,2,3,4,\dots)$.''

\section*{Branching out}

The next day, Queen Tanya gets a new test: $(1, 2, 3, 3)$.

Sage Siyona argues, ``Whoa! I see what's happening here. We know that $1$ is valid. By our definition, both $2$ and $3$ are meta-valid. This means we have $2$ distinct meta-valid answers! In our definition above, we are supposed to consider both of them together. But what if we branch out: add one at a time?''

Sage Julian makes a corresponding diagram in Figure~\ref{fig:branching-process}.
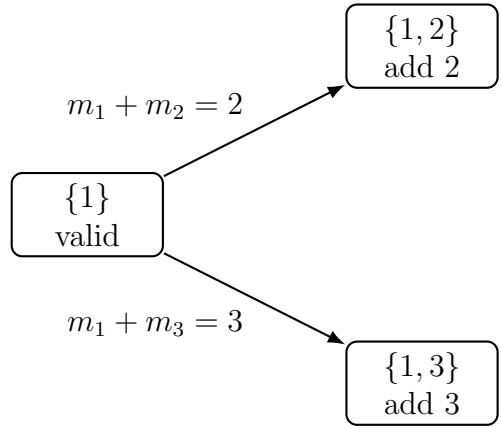
\begin{figure}[ht!] \centering 
\begin{tikzpicture}[ statenode/.style={rectangle, draw, thick, rounded corners, minimum height=0.85cm, minimum width=2cm, align=center}, arrow/.style={-Latex, thick} ] \node[statenode] (start) {$\{1\}$\\ valid}; \node[statenode, above right=1.1cm and 2.4cm of start] (two) {$\{1,2\}$\\ add $2$}; \node[statenode, below right=1.1cm and 2.4cm of start] (three) {$\{1,3\}$\\ add $3$}; \draw[arrow] (start) -- node[above left] {$m_1 + m_2 = 2$} (two); \draw[arrow] (start) -- node[below left] {$m_1 + m_3 = 3$} (three); \end{tikzpicture}
\caption{The test $(1,2,3,3)$ branches: starting from the valid answer $1$, both $2$ and $3$ can be added.} \label{fig:branching-process}
\end{figure}

Sage Ben has an idea: ``I think I invented a test with an infinite number of branches.''

\begin{example}
Consider the test $(1, 2, 2, 4, 5, 5, 6, 6, 6, 7, 7, 7, 7, \dots)$, where each answer $a\geq 4$ appears $a-3$ times. Evaluating this test, we see that the answers $1$ and $2$ are valid, and then each $a\geq 4$ is meta-valid and has its own branch.  
\end{example}

Sage Soham comments, ``This example makes sense only under the one-branch-at-a-time interpretation, while under the coexistence definition $A_1$ contains infinitely many answers, so $V_2$ is infinite.''

Sage Soham continues, ``We can generalize this example. We can start with any finite set of valid answers $a_1,\ldots, a_n$, each appearing $a_i$ times. Then we can add all the numbers $x$ greater than the validity score $V_1 =\sum_{i=1}^{n}a_i$ to the test, with each such number $x$ appearing $x-V_1$ times. Each such $x$ will create a separate branch.''

Sage Siyona suggests, ``Potentially, there are \textbf{TWO} approaches to branching out. We can follow individual branches, or we can combine branches together, and assume that all branches are valid and coexist together.''

Queen Tanya interrupts, ``We are halfway through our paper; let's not have too many new definitions at this point. Though it is hilarious: You are branching out the case of branching out!''

\section*{Para-Graphs}

The next day, Queen Tanya is resting. However, Sage Yonis interrupts, ``I have an idea to build a graph for para-tests. Note that we can make exactly $2^d$ subsets of a set of $d$ distinct answers. Our graph consists of $2^d$ vertices corresponding to these sets. Consider a vertex corresponding to a set $S$ with multiplicity $m_S$.''

``What is the multiplicity of a set?'' interrupts Sage Siyona. ``Sum of the multiplicities of its elements,'' answers Sage Yonis. ``For each set $S$, we draw a directed edge. If the multiplicity of a set is one of the answers $a$, then we draw an edge to the set consisting of one element $a$. If the multiplicity of the set is not one of the options, we draw an edge to the empty set. Figure~\ref{fig:paragraphs} describes the graph for Test $(1,2,2,3)$.''

\begin{figure}[ht!]
    \centering
    \includegraphics[scale=0.27]{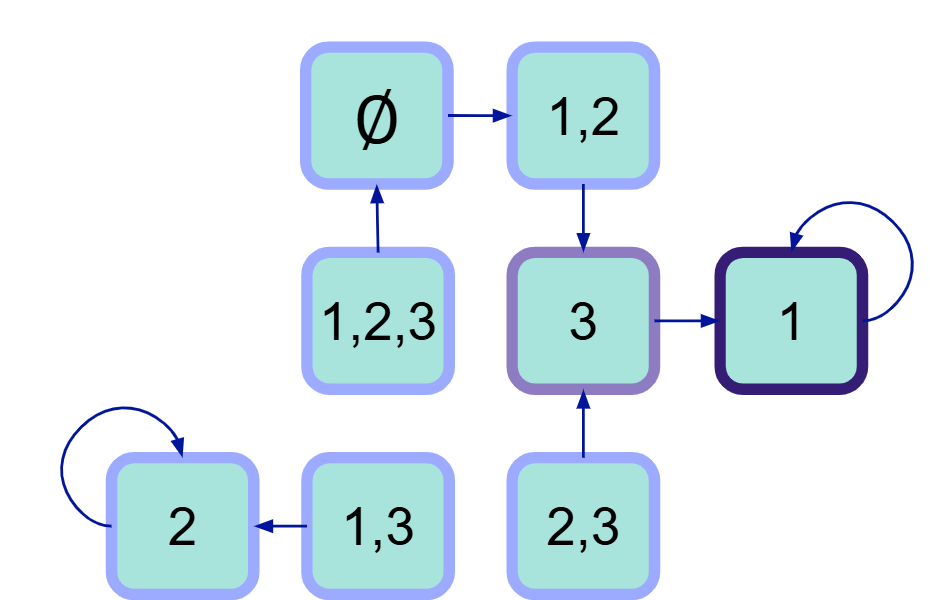}
    \caption{Para-graph for Test $(1,2,2,3)$}
    \label{fig:paragraphs}
\end{figure}

``What is the meaning of the edge from an empty set?'' Queen Tanya asks. ``This edge follows a special rule: It goes to the set of valid answers,'' Sage Siyona replies. ``If we have a para-test, the next edge always goes to the para-valid answer. If the path ends in a cycle of length $1$, the resulting answer is always valid, and is a good candidate for a `correct' answer.''

Sage Yonis notes, ``By the way, it is obvious that each vertex has exactly one outgoing edge. Moreover, if the set contains more than one answer, the edge goes to an empty set or a set with exactly one answer.''

Sage Siyona adds, ``While drawing directed graphs for different tests, I have come across an interesting fact: sometimes different tests have the same graph. For example, the tests $(1, 2, 2)$ and $(1, 3, 3, 3)$ correspond to isomorphic graphs.''

Queen Tanya argues, ``This example is not very illuminating. Both tests have two valid answers. It is not surprising that they have the same graph.''

\section*{Meta-Graphs}

Sage Siyona nods, ``I see, every IQ test can be represented as a directed graph. By the way, these graphs look different between para- and meta-tests.''

Sage Julian continues, ``Now, we consider meta-tests. We take a subset of answers $S$, and if there is an answer $a \notin S$ such that $m_S + m_a = a$, then we have an edge that points from set $S$ to set $S \cup \{a\}$. What is interesting is that each directed edge always points from a smaller subset to a larger subset. Moreover, the larger subset has to contain the smaller subset.''

Sage Ben says, ``Wow, it means that the graph has no cycles! Moreover, if we start with the subset of valid answers, we get an edge to the subset of valid and meta-valid answers, and so on.''

Sage Eric adds, ``You are forgetting that we have two possible
interpretations. If several answers can be added, we may add all of
them simultaneously, or we may create a separate branch for each one.
I made two pictures in Figure~\ref{fig:meta-graph} for the test
$(1,2,3,3)$. The left picture shows the coexistence interpretation,
in which $\{1\}$ leads to $\{1,2,3\}$. The right picture shows the
branching interpretation, in which $\{1\}$ leads separately to
$\{1,2\}$ and $\{1,3\}$.''

\begin{figure}[ht!]
    \centering
\includegraphics[scale=0.37]{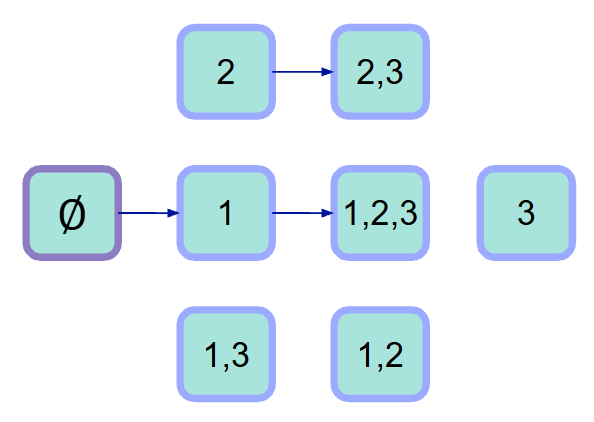} \qquad
\includegraphics[scale=0.37]{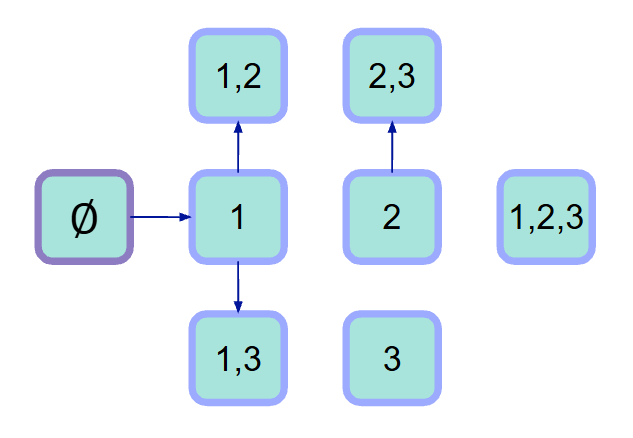}
    \caption{Two interpretations of the meta-process for $(1,2,3,3)$:
coexistence on the left and branching on the right}
    \label{fig:meta-graph}
\end{figure}

\section*{Definition Explosion}

After some thought, Sage Boya brings a suggestion: ``I want to suggest a million new definitions.''

\begin{definition}
We say that a test is
\begin{packed_item}
\item \textit{Solvable} if it has at least one valid answer.
\item \textit{Mono-solvable} if it has exactly one valid answer.
\item \textit{Meta-solvable} if it is a solvable test that has at least one meta-valid answer.
\item \textit{Meta-mono-solvable} if it is a mono-solvable test with exactly one valid answer and at least one other meta-valid answer.
\item \textit{Mono-meta-solvable} if it is a solvable test with valid answers that has exactly one meta-valid answer.
\item \textit{Mono-meta-mono-solvable} if it is a mono-solvable test with exactly one valid answer $a$ that has exactly one meta-valid answer.
\end{packed_item}
\end{definition}

Sage Eric is nitpicky: ``You exaggerated. You gave $6$ definitions, not a million.'' Sage Boya replies, ``I didn't exaggerate. I understated. We can continue this pattern of notation and definitions ad infinitum. See my example.''

\begin{example}
The test with answers $1,2,3,\dots,100$ is (mono-meta)$^{99}$-mono-solvable.
\end{example}

Queen Tanya stares at the board in silence, where $(\text{mono-meta})^{99}$-mono-solvable is written, and feels her blood pressure rise. The sages sense that something is off and avoid eye contact.

Sage Eric says, ``I see. I have an observation. An answer $a$ such that $m_a > a$ is never meta$^j$-valid.''

Sage Siyona comments, ``We have a meta-test example that branched out, but we do not have this issue with para-valid tests, as the para$^n$-valid answer is uniquely defined.''

Sage Boya says, ``Yes, this is why we don't have mono-para.''

\section*{The Test Becomes Self-Aware}

The next day, Queen Tanya gets a new test: $(1,2,2,3,4,4,4)$.

Sage Boya says, ``This looks like a meta-test, where the valid answer $2$, together with answer $3$ has a count of $3$, making $3$ meta-valid.''

Sage Eric continues, ``This example contradicts our definition, where we use all valid answers together to get a meta-valid answer. Let us give this test a different definition. What are synonyms of meta?''

Queen Tanya overhears the last question and blurts out, ``Facebook.''

``No-no,'' laughs Sage Eric, ``let us call it a self-aware test.''

\begin{definition}
A \textit{self-aware} test is a test for which there exists at least one non-empty subset $\{a_1,a_2,\dots,a_k\}$ of valid answers and another answer $c$ such that $c$ appears $c-\sum_{i=1}^ka_i$ times.
\end{definition}

Sage Boya adds, ``Remember Test 6: $(1,2,2,3)$, a paradoxical test we looked at a long time ago. It is actually also a self-aware test when $k=1$, $a_1=2$, and $c=3$. The cheapest paradoxical test that is not a self-aware test is $(1,2,2,3,3,3)$, with a cost of $14$.'' ``No-no,'' says Sage Yonis, ``This test is not paradoxical: its valid answers are $1$, $2$, and $3$, so its validity score is $6$, which is absent. The cheapest paradoxical test that isn't self-aware is $(2,2,3,3,3,5)$ with the cost of $18$.'' ``I don't agree,'' counters Sage Karam, ``The cheapest paradoxical test that isn't self-aware is $(1,2,2,3,3,3,3)$ with the cost of $17$.'' Everyone agrees with Karam.

Sage Boya opens his laptop. The other sages take this as a warning that a sequence is coming. Sage Boya writes a program to find the number of self-aware tests given the cost. The resulting sequence starting from index $0$:
\[0,\ 0,\ 0,\ 1,\ 0,\ 0,\ 1,\ 3,\ 2,\ 4,\ 3,\ 7,\ 9,\ 14,\ \dots.\]
%0, 0, 0, 1, 0, 0, 1, 3, 2, 4, 3, 7, 9, 14, 15, 26, 26, 38, 47, 69, 79, 115, 132, 178, 210, 274, 329, 436, 519, 662, 792, 1003, 1189, 1493, 1775, 2207, 2631, 3236, 3837, 4707, 5562, 6748, 7990, 9641, 11379, 13663, 16090, 19229, 22614, 26896, 31540, 37397, 43758, 51644, 60333, 70971, 82740, 97042, 112878, 132002, 153263, 178721, 207083, 240903, 278595, 323222, 373186, 431968, 497862, 575057, 661589, 762583, 875901, 1007534, 1155390, 1326604, 1518866, 1740800, 1990044, 2276911, 2599046, 2968813, 3383916, 3859230, 4392792, 5002071, 5685887, 6465154, 7339219, 8333318, 9447848, 10712981, 12130741, 13736997, 15535997, 17570867, 19848657, 22420619, 25298179, 28542510

Sage Brandon notices, ``A meta-test is also a self-aware test, implying that for every cost $C$, the number of self-aware tests has to be the same or greater than the number of meta-tests.''

Sage Karam continues, ``Since every meta-test is self-aware, the first genuinely new examples are self-aware tests that are not meta-tests. The cheapest self-aware test that is not meta should be at a cost of $8$, and the next one at a cost of $11$.''

\begin{example}
The cheapest self-aware test that is not a meta-test is $(1,2,2,3)$, which has cost $8$. The next such example is $(1,2,2,3,3)$ with cost $11$.
\end{example}

Sage Siyona says, ``I have an example where self-awareness leads to two different paths.''

\begin{example}
Consider the test $(1,2,2,4,4,5,5,5,5)$. The valid answers are $1$ and $2$.
If we use only the valid answer $1$, then adding the answer $5$ gives a total
multiplicity of $5$, since $1+4=5$. If we use only the valid answer $2$, then
adding the answer $4$ gives a total multiplicity of $4$, since $2+2=4$.
Thus, self-awareness can lead to two different paths.
\end{example}

\section*{Temporal validity}

Sage Brandon argues, ``I am not happy with how we treat meta-tests. I think we should introduce time steps and states that depend on time.''

\begin{definition}
    A \textit{state} is a set of answer values that we choose to be valid at a particular step.
\end{definition}

Here is how the iteration works. Consider a subset of answers $S$ generated at some step, with the total multiplicity of the answers $m_S$. For each answer $a$ such that $S \cup \{a\}$ has total multiplicity $a$, the next state contains $a$. If no such answer $a$ exists, the next state is the empty set. Here $a$ is allowed to belong to $S$. If $a\in S$, then $S\cup\{a\}=S$, so the condition reduces to $m_S=a$.

\begin{example}
   Consider the test $(1,3,3,5,5,7)$. Let us start with the subset $\{1,3\}$. The answers that satisfy the condition are $3$ and $5$. Therefore, the next state is $\{3,5\}$.
\end{example}

We can make a graph with such iterations. Then, to understand how to treat valid and correct answers, we start with the empty set.

Sage Soham interrupts, ``I see, the next state is the set of valid answers.''

\begin{example}
   Consider test $(1,3,3,5,5,7)$. We start with the empty set. The next set is the set of valid answers: $\{1\}$. After that, we get $\{1,3\}$, followed by $\{3,5\}$. Then we get to the empty set.
\end{example}

Sage Julian adds another diagram in Figure~\ref{fig:temporal-loop}.
\begin{figure}[ht!] \centering \begin{tikzpicture}[ statenode/.style={rectangle, draw, thick, rounded corners, minimum height=0.85cm, minimum width=1.8cm, align=center}, arrow/.style={-Latex, thick} ] \node[statenode] (empty) {$\emptyset$}; \node[statenode, right=2.1cm of empty] (one) {$\{1\}$}; \node[statenode, right=2.1cm of one] (onethree) {$\{1,3\}$}; \node[statenode, right=2.1cm of onethree] (threefive) {$\{3,5\}$}; \draw[arrow] (empty) -- (one); \draw[arrow] (one) -- (onethree); \draw[arrow] (onethree) -- (threefive); \draw[arrow, bend left=35] (threefive) to node[below] {no continuation} (empty); \end{tikzpicture} \caption{The temporal-validity process for $(1,3,3,5,5,7)$ gives a loop through four states.} 
\label{fig:temporal-loop} \end{figure}
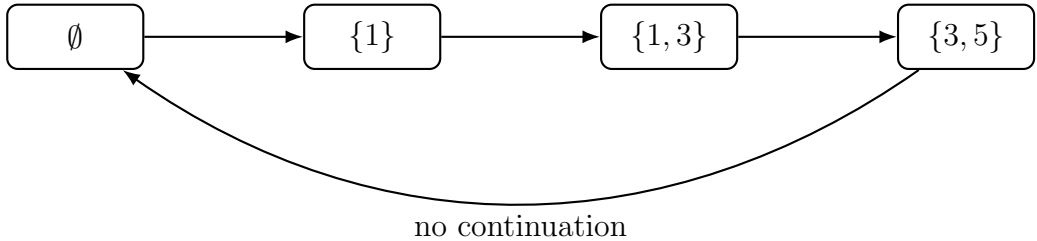

Sage Brandon, never one to let a paradox go unnamed, immediately proposes a definition: ``We are in a loop of length $4$. So, I call it a $4$-looping test.''

Sage Siyona says, ``This is so confusing. Let me try a different example.''

\begin{example}
   Consider $(1,2,2,4)$. We start with the empty set. The next set is the set of valid answers: $\{1,2\}$. After that, we get $\{4\}$. Then we get to the empty set. Thus, it is a $3$-looping test. We can show this as a graph in Figure~\ref{fig:metagraph}.
   \begin{figure}[ht!]
       \centering
    \includegraphics[scale=0.4]{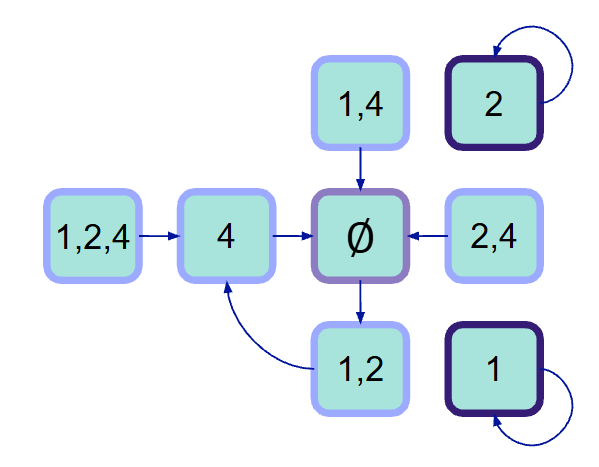}
    \caption{Meta-graph with temporal validity for Test $(1,2,2,4)$}
    \label{fig:metagraph}
   \end{figure}
\end{example}

Sage Ben explains, ``Similar to graphs for para-tests, in graphs for the meta-tests, the edge from an empty set also goes to the set of valid answers, and the next edge always goes to the set of meta-valid answers.'' Sage Yonis counters, ``No-no. This happens if the test is polysolvable. If the test is monosolvable, then the valid answer $a$ stays, and the next state contains $a$ together with the meta-valid answers.'' Sage Ben continues, `` I see. In the case of a path ending in a cycle of length $1$, the resulting state is also always valid and would similarly be a good choice for a 'correct' answer.'' Sage Eric notes, ``Like in para-graphs, each vertex has exactly one outgoing edge. However, unlike in para-graphs, an arrow from a non-empty set does not always have to point to an empty set or a set with exactly one answer. For example, in the test $(1, 2, 2, 4, 5, 5)$, the edge coming from the set $\{1, 2\}$ would go to the set $\{4, 5\}$.''

Sage Julian says, ``Does the loop always come back to the empty set?'' Sage Eric replies, ``No, I have come up with an example.''

\begin{example}
    Consider a test $(2, 2, 5, 5, 5, 6, 6, 7, 7, 7, 7)$. We start out as $\emptyset$, then $\{2\}$, then $\{2, 5\}$, then $\{5\}$, then $\{7\}$, then $\{6\}$, and then back to $\{5\}$. This means the test loops between $\{5\}$, $\{7\}$, and $\{6\}$. 
\end{example}

Sage Karam comments, ``We might also never get into a loop.'' Eric comes up with an example again.

\begin{example}
Consider the infinite test in which $1$ appears once, $2$ appears twice, and every odd answer $2n+1\geq 7$ appears $n+1$ times:
\[
(1,\ 2,\ 2,\
\underbrace{7,\ldots,7}_{4\text{ times}},\
\underbrace{9,\ldots,9}_{5\text{ times}},\
\underbrace{11,\ldots,11}_{6\text{ times}},\ldots).
\]
We leave it to the reader to check that the temporal-validity process is
\[
\emptyset\longrightarrow \{1,2\}\longrightarrow
\{7\}\longrightarrow\{9\}\longrightarrow\{11\}\longrightarrow\cdots.
\]
Therefore, the states never repeat, and the test is loopless.
\end{example}

\section*{Tanya's Headache}

Tanya gets a headache. ``What have I started? They have so many ideas, I can't process all of them.''

Sage Leone suggests, ``What if we \textbf{ask about incorrect answers}?'' He continues, ``This question is nuanced enough to create the liar paradox with just one answer choice.''

\begin{center}
How many incorrect answers are there?

\textbf{(a)} 1
\end{center}

Boya analyzes it, explaining, ``If I believe (a) is correct, then by its statement, (a) must also be incorrect. If I believe (a) to be incorrect, then the total number of incorrect answers becomes 1, and that makes (a) correct. This creates a liar paradox.'' Boya also suggests another test.

\begin{center}
How many incorrect answers are there?

\textbf{(a)} 2 \quad \textbf{(b)} 0
\end{center}

But other sages have so many ideas of their own that they decided to think about this test later. Sage Yonis has a different idea: ``What if \textbf{we allow zero}? Consider a test $(1,1,2)$. Then it has $0$ valid answers, and $0$ appears $0$ times. So the answer $0$ must be valid for this test, and this test is solvable!''

Queen Tanya interrupts, ``But the correct answer has to appear in the text, right? I am so confused.''

Sage Karam then has another interesting idea. ``What if,'' he says, ``we had tests with \textit{partially valid answers}?'' Sage Yonis comments, ``This is too complicated. What if instead of integers we allow \textbf{any real number}?'' Sage Ben continues, ``Let's start with a simpler idea. What if we \textbf{replace $100$ percent with any other value}?''

Sage Ben says, ``We can do many interesting things even before replacing $100$ with another number. For example, having an answer of more than 100\% in a percentage test doesn't make much sense. What if we \textbf{disallow answers that are greater than 100 in percentage tests}?'' Sage Eric counters, ``We can add a new parameter $b$ and \textbf{disallow answers that are greater than $b$ in percentage tests}.''

Sage Brandon suggests, ``What if, when we have a branch of $n$ possibilities, we give each one a \textbf{weight} of $\frac{1}{n}$?''

Sage Karam moves in a different direction, ``What if we allow a \textbf{non-integer number of options}?'' Sage Yonis adds, ``What if we \textbf{assign weights to the options}?''

Sage Brandon is deep in his thoughts, ``Currently, a self-aware test is an attribute that can be given to a test. I think that self-aware tests should instead be a \textbf{modifier} to how a test is solved, like temporal validity. This means that there are self-aware meta-tests, as well as self-aware para-tests, and the temporal validity versions.''

Sage Karam continues, ``What if the answers are not fixed numbers, but rather \textbf{probability distributions}? What if the cost is \textbf{dependent on the number of sages}?''

Tanya stops listening.

\section{Acknowledgments}

We are grateful to the PRIMES STEP program for giving us the opportunity to conduct this research. After the paper was written, we asked ChatGPT to check for mistakes.

\section{AI Usage}

After the paper was written, we used AI to check for mistakes.

\end{document}